\documentclass[11pt,reqno, english]{amsart}
\usepackage[a4paper]{geometry}
\usepackage{amsmath,amssymb,amscd,amsthm,amsfonts}
\usepackage{graphicx,subcaption}
\usepackage{hyperref}
\usepackage{dsfont}
\usepackage{xcolor}
\usepackage{color,soul}
\usepackage[utf8]{inputenc}
\usepackage[nobysame, alphabetic]{amsrefs}
\usepackage{tikz}
\usepackage[capitalise]{cleveref}

\newtheorem{theorem}{Theorem}[section]

\newtheorem{lemma}[theorem]{Lemma}

\newtheorem{conjecture}[theorem]{Conjecture}

\theoremstyle{definition}

\newcommand{\rr}{\mathds{R}}
\newcommand{\zz}{\mathds{Z}}

\title{Fair distribution of bundles}

\author[Sober\'on]{Pablo Sober\'on}\address{Baruch College \& The Graduate Center, City University of New York, New York, NY 10010} 
\email{psoberon@gc.cuny.edu}



\thanks{The research of P. Sober\'on was supported by NSF CAREER award no. 2237324 and a PSC-CUNY Trad B award.}

\begin{document}

\begin{abstract}
In this paper, we study the problem of splitting fairly bundles of items.  We show that given $n$ bundles with $m$ kinds of items in them, it is possible to distribute the value of each kind of item fairly among $r$ persons by breaking apart at most $(r-1)m$ bundles.  Moreover, we can guarantee that each participant will receive roughly $n/r - mr/2$ full bundles.  The proof methods are topological and use a modified form of the configuration space/test map scheme.  We obtain optimal results when $r$ is a power of two.
\end{abstract}

\maketitle

\section{Introduction}

Problems related to fair distribution of resources often showcase the applications of topological methods to combinatorics.  This may seem natural in the solution of distribution problems with a geometric setting \cites{Zivaljevic2017, RoldanPensado2022}; yet there are some purely combinatorial problems in which these techniques are useful.  There is also a vast literature on fair distribution of divisible or of indivisible goods \cite{Liu2024} given the relevance of these results to mathematical economics.  In this paper, we deal with distribution problems where the goods are divisible, but we aim to divide as few of them as possible, in addition to distributing the undivided bundles of goods as evenly as possible.  Let us first mention some emblematic results in this area.

A classic example is the \textit{necklace-splitting theorem}.  In the usual setting, $r$ thieves steal an open necklace with $m$ kinds of gemstones, and they want to distribute the gemstones of each kind fairly among themselves by cutting the necklace into as few intervals as possible.  Alon showed that such a fair distribution can always be achieved by cutting the necklace at most $(r-1)m$ times \cite{Alon1987}, which is optimal.  The first solutions when $r$ has odd prime factors were topological topological.  There are new non-topological proofs \cites{Meunier2014, FilosRatsikas2021}.  There are several solutions to this problem for $r=2$ \cites{Hubard2024, Hobby1965, Alon1986, Goldberg1985}.  Another key example is the \textit{cake-cutting theorem}.  In this problem, the $[0,1]$ interval is cut into $r$ intervals to be distributed among $r$ persons.  The main result is that, with very light conditions on the individual preferences of the persons, an envy-free distribution always exists (i.e., no person prefers the piece that someone else received) \cites{Stromquist1980, Woodall1980}.

In both the necklace-splitting theorem and the cake-cutting theorem, we have an order imposed on the objects to be distributed.  Indeed, if we wanted to distribute the gemstones in the necklace-splitting theorem without the order condition, we can simply give each thief their corresponding number of each kind of gemstone.  There are many variations and generalizations of both the necklace-splitting theorem, the cake-cutting theorem, and the tools used in their solutions at the forefront of current research (see, e.g., \cites{procaccia2015cake, Blagojevic2018, McGinnis2024a, FilosRatsikas2021, Avvakumov2021, Asada:2018ix} and the references therein).  

In this paper, we study another combinatorial distribution problem in which topological tools provide a solution.  In this problem we do not have an induced order on the objects.  We aim to split $n$ objects with $m$ different valuations assigned to them (or $n$ bundles with $m$ different kinds of products in them) by cutting/breaking as few as possible and by giving each participant as many full bundles as possible.  Following tradition in the area, we give a lighthearted interpretation: children distributing cookies.  Our main results is as follows.

\begin{theorem}\label{thm:main-theorem}
    Let $r, n, m$ be positive integers.  A group of $r$ children has $n$ cookies.  Each cookie has some amount of $m$ possible kinds of frosting.  Then, it is possible to cut at most $(r-1)m$ cookies and distribute all the cookies so that each child has the same amount of each kind of frosting and at least $\displaystyle \left\lfloor \frac{n-\left\lceil\frac{m(r-1)}{2}\right\rceil}r\right\rfloor-\left\lfloor\frac{m(r-1)}{2}\right\rfloor$ full cookies.
\end{theorem}

Recently, Joji\'c, Panina, and \v{Z}ivaljevic proved several generalizations of Alon's necklace-splitting theorem \cite{Jojic2021}.  One of their results, \textit{the equi-cardinal necklace splitting theorem} shows that we can cut necklaces with $m$ measures as in the necklace-splitting theorem with the additional constraint that each thief receives roughly the same number of intervals of the cut.  It seems that neither \cref{thm:main-theorem} nor Joji\'c et al.'s result imply the other, and the use of topological tools in the proofs are different.

The number of cuts  in \cref{thm:main-theorem} is optimal, as we could have the frosting of kind $j$ distributed evenly among $r-1$ cookies.  Each of those cookies has to be cut, and we can do this for each kind of frosting without repeating any cookies.

The number of cookies that we can guarantee that each child receives can be at most $\displaystyle \left\lfloor \frac{n-m(r-1)}r\right\rfloor = \left\lfloor \frac{n+m}{r}\right\rfloor - m$, as that would be an equitable distribution of all full cookies.  We conjecture that this should be possible.

\begin{conjecture}
 \label{conj:main-conj}
    Let $r, n, m$ be positive integers.  A group of $r$ children has $n$ cookies.  Each cookie has some amount of $m$ possible kinds of frosting.  Then, it is possible to cut at most $(r-1)m$ cookies and distribute them so that each child has the same amount of each kind of frosting and at least $\displaystyle \left\lfloor \frac{n-m(r-1)}r\right\rfloor$ full cookies.
\end{conjecture}

In \cref{sec:two-children}, we confirm \cref{conj:main-conj} for the case $r=2$, and we give an algorithmic solution for $m=r=2$.  In \cref{sec:remarks}, we explain how the case $r=2$ implies the case $r=2^k$ when $n+m$ is a multiple of $r$.

\begin{theorem}\label{thm:two-optimal}
    Let $n,m$ be positive integers.  Two children have $n$ cookies.  Each cookie has some amount of $m$ possible kinds of frosting.  Then, it is possible to cut at most $m$ cookies and distribute them so that each child has the same amount of each kind of frosting and at least $\left\lfloor \frac{n-m}{2}\right\rfloor$ full cookies.
\end{theorem}

\begin{theorem}\label{thm:powers-of-two}
    Let $n,m,t$ be positive integers and let $r=2^t$.  Suppose $n+m$ is a mulitple of $r$.  A group of $r$ children has $n$ cookies.  Each cookie has some amount of $m$ possible kinds of frosting.  Then, it is possible to cut at most $(r-1)m$ cookies and distribute them so that each child has the same amount of each kind of frosting and at least $\frac{n-(r-1)m}{r}$ full cookies.  If $n+m$ is not a multiple of $r$, we can guarantee that each child receives at least $\left\lfloor \frac{n-(r-1)m}{r}\right\rfloor - (r-2)$ full cookies.
\end{theorem}

Our proof of \cref{thm:main-theorem} has a few additional benefits.  If a cookie is to be distributed among $k$ children, we have to use at least $k-1$ cuts for it.  In other words, once we decide to cut a cookie we still cannot distribute its contents in any way unless we spend more cuts.  Additionally, if we want to have some geometry in the problem and think of the cookies as objects in $\rr^3$ with measures on them (or even as objects in $\rr^d$), we do not need elaborate cuts.  The distribution can be made by using cuts by hyperplanes.

The main idea of the proof is to use the test map/configuration space scheme, which is a standard approach in topological combinatorics.  We parametrize the space of partitions and construct a test map that checks whether a partition meets the conditions we want.  The construction of the test map involves some nuance to force an equitable distribution of the full cookies.  Readers familiar with the test map scheme may initially be drawn to the following approach: use a similar configuration space as with the necklace-splitting theorem and adjust the test map to check that the number of full cookies is also fairly distributed.  Although this approach gives a simple proof that the children can expect to receive $\frac{n}{r}-O(mr)$ full cookies, it requires $(r-1)(m+1)$ cuts and guarantees a smaller number of full cookies for each child than \cref{thm:main-theorem}.  In \cref{sec:remarks} we describe this approach.

Our construction of the configuration space also involves a nonstandard technique, using part of a construction that Vu\v{c}i\'c and \v{Z}ivaljevi\'c used to prove lower bounds for Sierksma's conjecture \cite{Vucic1993} around Tverberg's theorem (see, e.g., \cite{Barany2018}).  Without the Vu\v{c}i\'c--\v{Z}ivaljevi\'c construction, our method would guarantee that each child receives at least $\lceil\frac{n}{r}\rceil-m(r-1)$ full cookies.  \cref{thm:main-theorem} is half-way between this bound and the one in \cref{conj:main-conj}.

Some problems related to \cref{thm:main-theorem} have appeared in mathematical contests \cite{BKC05}.  In those problems, we want to satisfy the preferences of a single person while taking as few cookies as possible (see \cite{RoldanPensado2022}*{Ex. 2.4.4}, but instead require that a person wants to receive at least a $(1/r)$-fraction of each kind of frosting).  Solving such problems is much more straightforward; it suffices to place the cookies in a line, apply the necklace splitting theorem, and look for the person who received parts from the smallest number of cookies.  Once we want to satisfy more people, the modifications to the test map scheme become relevant.

In \cref{sec:notation} we present some preliminary tools and establish the notation for our proofs.  Then, in \cref{sec:two-children} we prove \cref{thm:two-optimal}.  We describe the intuition behind our proof in \cref{sec:intuition} before proving \cref{thm:main-theorem} in \cref{sec:main-proof}.  Finally, in \cref{sec:remarks} we prove \cref{thm:powers-of-two} and give additional remarks on the problems and methods from the paper.

\section{Notation and definitions}\label{sec:notation}

To prove \cref{thm:main-theorem} we need to work with a special family of sets of measures.  We say that $\mu$ is a \textit{labeled partition} if, for some positive integer $n$, it is an ordered $n$-tuple of measures on $\rr^1$, namely $\mu=(\mu_1,\dots,\mu_n)$. 

We say that a topological space $X$ is $n$-connected if its reduced homology groups $\Tilde{H}_k(X)$ with rational coefficients vanish for all $k \le n$.  We also say that a disconnected but non-empty set is $(-1)$-connected.  A path-connected set is $0$-connected.  A direct consequence of the K\"unneth formula for joins is that if $X$ and $Y$ are topological spaces such that $X$ is homologically $n$-connected and $Y$ is homologically $m$-connected, then their join $X * Y$ is homologically at least $(n+m+2)$-connected.  In particular, the $(s+1)$-fold join of a discrete set is at least $(s-1)$-connected, and the $(s+1)$-fold join of a connected set is at least $2s$-connected.

The main topological tool we need is the following generalization of the Borsuk--Ulam theorem by Dold.  Given two spaces $X$, $Y$ with an action of some group $G$, a function $f:X\to_G Y$ is $G$-equivariant if $f(gx) =gf(x)$ for all $x \in X$ and $g\in G$.  

\begin{theorem}[Dold 1983 \cite{Dold:1983wr}]\label{thm:dold}
Let $G$ be a finite group, $|G|>1$, $X$ be an $n$-connected space with a free action of $G$, and $Y$ be a paracompact topological pace of dimension at most $n$ with a free action of $G$.  Then, there exists no $G$-equivariant continuous map $f:X \to_G Y$. 
\end{theorem}

We denote by $S^m\subset \rr^{m+1}$ the unit sphere centered at the origin.

\section{Simple proofs for \cref{thm:two-optimal}}\label{sec:two-children}

A careful analysis of the proof of \cref{thm:main-theorem} in \cref{sec:main-proof} shows that it also yields a proof of \cref{thm:two-optimal} with the optimal constants.  This is because when there are two children, there are no \textit{bad cuts} as we will define in \cref{sec:intuition}.  However, we can present a simpler self-contained proof of this case.  We use the following Lemma by Gale:

\begin{lemma}[Gale 1956 \cite{Gale1956}]\label{lem:Gale}
    Let $n',m$ be positive integers.  There exists a set of $2n'+m$ points on $S^m$ such that every hyperplane through the origin in $\rr^{m+1}$ leaves at least $n'$ points of the set in each of its two open half-spaces.
\end{lemma}

\begin{proof}[Proof of \cref{thm:two-optimal}]
    If $n+m$ is odd, we can add an empty cookie to the set.  Then, the numer of cookies is $2n'+m$ for some $n'$.  Construct a set $S \subset S^m$ of cardinality $2n'+m$ as in \cref{lem:Gale}.  We pair $S$ with the set of cookies, and assign to each point of $S$ a total of $m$ labels: the amount of frosting of each kind that its corresponding cookie has.  We can consider the set $S$ with the labels corresponding to a single kind of frosting as a discrete measure.  Therefore, we now have $m$ measures on $\rr^{m+1}$.  Take an additional discrete measure concentrated at the origin in $\rr^{m+1}$.

    Apply the ham sandwich theorem to these $m+1$ measures on $\rr^{m+1}$.  We obtain a hyperplane $H$ through the origin such that its two closed half-spaces have at least half of each kind of measure (we double-count the weights of the points contained in $H$).  The cookies corresponding to the points in $H$ are those we may need to cut to balance the amount of frosting of each kind.  Due to \cref{lem:Gale}, each child receives at least $n'$ full cookies.  Note that, if we added an empty cookie at the start, the number of original full cookies a child receives is at least $n'-1 = \lfloor\frac{n-m}{2}\rfloor$.  If we did not add an empty cookie, then $n' = \frac{n-m}{2}$.
\end{proof}

The proof of Gale's lemma is not hiding any advanced technique under the rug.  It requires us to place the points in a moment curve in a hyperplane in $\rr^{m+1}$ an project radially to $S^m$ using an alternating argument (see \cite{Matousek2003}*{p.65}).  We also show a non-topological proof of \cref{thm:two-optimal} for $m=2$ which gives us an algorithm to find a distribution.

\begin{proof}
    As in the previous solution, assume that the number of cookies is $2n'+2$ without loss of generality.  Suppose the cookies have black frosting and white frosting.  Let $C_1$ be the cookie with the most black frosting, and let $C_2$ be the cookie with the most white frosting among the remaining cookies.  We show that we can make a fair distribution by cutting at most $C_1$ and $C_2$.

    Order the rest of the $2n'$ cookies and to cookie $i$ assign the pair $(x_i,y_i)$, where $x_i$ is the amount of black frosting it has.  Assume without loss of generality that $x_1 \ge x_2 \ge \dots \ge x_{2n}$.  Consider the pairs
    $\{(x_1,y_1),(x_2,y_2)\}$, $\{(x_3,y_3),(x_4,y_4)\}$, $\dots $, $\{(x_{2n'-1},y_{2n'-1})$, $(x_{2n'},y_{2n')}\}$.  We will give one cookie in each pair to each child.  This way, each child will receive $n'$ full cookies.  Let us first show that, regardless of how we do this distribution, the difference in black frosting between the two children cannot exceed $x_1$.  In particular, by cutting $C_1$ we can balance the amount of black frosting each child receives.

    The largest difference we could have in black frosting is achieved by giving one child the cookie with the most black frosting in each pair.  The difference would be
    \begin{align*}
    (x_1-x_2) + (x_3-x_4) + \dots + (x_{2n'-1}-x_{2n'}) & = \\
    x_1-(x_2 - x_3)-(x_4-x_5) - \dots - (x_{2n'-2}-x_{2n'-1})-x_{2n'} & \le x_1.
    \end{align*}

    Therefore, the difference in black frosting is in the interval $[-x_1,x_1]$, as we wanted.  It suffices to show that there is a way to make the choices in the pairs so that the difference in white frosting is in the interval $[-\max_{i}y_i,\max_iy_i]$.  For $i=1,\dots,n'$, let $b_i = y_{2i-1}-y_{2i}$ and let $c_1 \ge c_2 \ge \dots \ge c_{n'}$ be a permutation of $b_1,\dots,b_{n'}$.  By the same argument, the value of the alternating sum $c_1-c_2+c_3 - \dots + (-1)^{n'}c_{n'}$ is in the interval $[0,c_1]$, and $c_1 \le \max_i y_i$.  Therefore, by making the choice that gives us this difference in white frosting we can cut $C_2$ to balance the white frosting. 
    \end{proof}

\section{Intuitive idea of the main proof and distribution arguments}\label{sec:intuition}

The main idea behind the proof will be the following.  We first label the cookies from $1$ to $n$ in some order.  We start by giving cookie $1$ to child $1$ and we will start to assign them in a cyclic order modulo $r$: the next cookie goes to child $2$, the next cookie goes to child $3$ and so on.  At some point, if we are to give a cookie to child $i$, we may instead cut that cookie and give part of it to child $i$ and part of it to child $j$.  Then we update our counting and the next cookie goes to child $j+1$.  We continue to do this, allowing ourselves to cut at most $(r-1)m$ cookies.  

\begin{figure}
    \centering
    \includegraphics[]{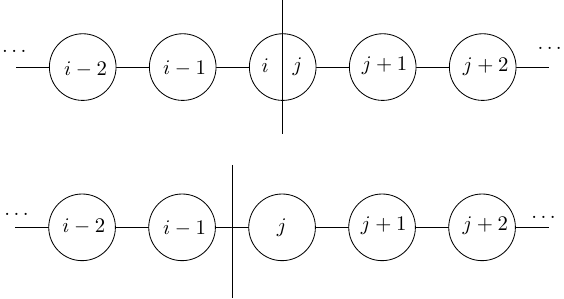}
    \caption{A representation of the possible cuts.  If $j=i$, it is equivalent to not cutting the cookie, so it can be ignored.  If $j=i-1$, no child is being skipped.  If we wanted to change our index of which child gets the next cookie (as in the bottom image), we have a situation equivalent to ``cutting'' the cookie with label $j$ and giving nothing to child $i$.  This situation of cutting between cookies may happen in \cref{sec:main-proof}, but it does not affect the analysis of this section.}
    \label{fig:enter-label}
\end{figure}

In the next section, we will construct a topological space that parametrizes the set of all possible partitions, and prove that for some partition all children receive the same amount of some kind of frosting.  In the rest of this section, we prove that a partition as described gives each child the desired number of full cookies, provided one condition on the cut cookies is met.

If we never cut cookies, each child is receiving $\lfloor n/r\rfloor$ cookies, which would be an ideal scenario.  If at some point we are cutting a cookie but $i=j$ in the scenario above, then we can ignore the cut, keep that cookie whole, and continue as before.  If $j=i-1$, then the next full cookie goes to child $i$, which means that, excluding the cut cookie, child $i$ received the next full cookie.  If we only have this kind of cuts, we would ultimately give each child $\lfloor \frac{n-(r-1)m}{r}\rfloor$ cookies, precisely the number from \cref{conj:main-conj}.  Thus, every cut in which $j \neq i$ and $j \neq i-1$ makes us skip at least child $i$, so they could potentially receive one fewer cookie than the rest.  

Let us call these the \textit{bad cuts}.  Imagine that we start to distribute the cookies and each child has an initial expectation to receive $\lfloor \frac{n-(r-1)m}{r}\rfloor$ cookies.  If we ignore the amount of frosting on each cookie, a bad cut is essentially the same as taking the cookies that should go to children $i,i+1,\dots, j$ (cyclically modulo $r$) and moving those to the end.  This way, the first cookie after the cut goes to $j+1$, and the cookies that were moved to the end will be distributed according to our algorithm (some may go to these children).  Therefore, if we consider $\overline{j+1-i} \in \zz_r = \{0,1,\dots,r-1\}$ the congruence of $j+1-i$ modulo $r$, the maximum possible number of cookies received for the children skipped goes from $\lfloor \frac{n-(r-1)m}{r}\rfloor$ to $\lfloor \frac{n+(\overline{j-i})-(r-1)m}{r}\rfloor -1 \ge \lfloor \frac{n+1-(r-1)m}{r}\rfloor -1$ due to this bad cut.

If there are $b$ bad cuts, the worst case scenario is that we always skipped the same child $i$ and they would receive $\lfloor\frac{n+b-(r-1)m}{r}\rfloor -b$ full cookies at the end.  In the next section, the distributions we make will have at most $\lfloor \frac{(r-1)m}{2}\rfloor$ bad cuts.  This implies the guaranteed number of full cookies from \cref{thm:main-theorem}: 
\[
\left\lfloor \frac{n-\left\lceil\frac{m(r-1)}{2}\right\rceil}r\right\rfloor-\left\lfloor\frac{m(r-1)}{2}\right\rfloor.
\]

\section{main proof}\label{sec:main-proof}

We first prove \cref{thm:main-theorem} when $r$ is a prime number.  For convenience, we will write $p$ for $r$ when we assume it is prime.

To parametrize our space of partitions, we use a construction inspired by one used by Vu\v{c}i\'c and \v{Z}ivaljevi\'c in their proof of a bound for Sierksma's conjecture  \cite{Vucic1993}.  This has also been used recently to count Tverberg partitions for polytopes \cite{soberon2024improved}.  Consider $G=\zz_p$ as a discrete set with $p$ points.  The join $G*G$ is a complete bipartite graph, $K_{p,p}$.  Instead of using this set in our construction, we will only use a cycle in it.  Formally, we will use the space

\[G_{*2}= \{\alpha x + (1-\alpha) y \in G*G:x-y \equiv 0 \mod p \mbox{ or } x-y \equiv 1 \mod p\}\]

\begin{figure}
    \centering
    \includegraphics[]{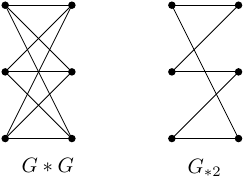}
    \caption{An illustration of $G*G$ and $G_{*2}$ for $G= \zz_3$.}
    \label{fig:joins}
\end{figure}

Note that $G_{*2}$ is a connected $G$-space.  Intuitively, we drop as many edges as possible from $G*G$ while maintaining a free action of $G$ and the connectedness of the space, see \cref{fig:joins}.  This will allow us to have much more control over the resulting partition, which will increase the number of full cookies that we can guarantee that each child has by roughly $\frac{1}{2}(p-1)m$.

\begin{proof}[Proof of \cref{thm:main-theorem} for $r$ prime]
    Represent each cookie as an open subinterval in the interval $[0,1]$ without any overlaps.  For $j=1,\dots,k$, we can assign to the $i$-th interval corresponding to a cookie a uniform measure $\mu^j_i$ according to how much frosting of type $j$ the cookie has.  This gives us a labeled measure $\mu^j = (\mu^j_1,\dots, \mu^j_n)$.

    Let $G=\zz_p=\{0,1,\dots,p-1\}$ and $s=(p-1)m$.  The $(s+1)$-fold join $G^{*(s+1)}$ parametrizes the set of labeled partitions of $[0,1]$ into $s+1$ intervals, where each interval has a label in $\zz_p$.  \cref{fig:labels} illustrates this parametrization.  We want a sparser set of partitions, so we restrict ourselves to a particular subset.  If $s$ is odd, we instead use $K=(G_{*2})^{*((s+1)/2)}$.  If $s$ is even, we use $K=(G_{*2})^{*(s/2)}*G$.  Given a labeled partition parametrized by $K$, for each $i \in [p]$, let $A_i$ be the subset of $[0,1]$ of all points with label $i-1$.  Note that $G_{*2}$ corresponds to the partitions of $[0,1]$ into two intervals, each labeled with an element of $\zz_p$.  However, if the label of the left interval is $i$, then the label of the right interval can only be $i$ or $i-1$.  Therefore, in the distribution below, the number of copies of $G_{*2}$ we use in $K$ corresponds to the number of cuts which are not bad cuts.

    \begin{figure}
        \centering
        \includegraphics{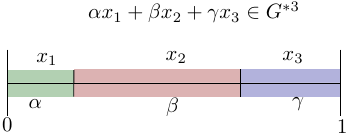}
        \caption{This figure shows how an element of $G^{*3}$ corresponds to a partition of $[0,1]$ into three labeled intervals.  The coefficients $\alpha, \beta, \gamma$ of the convex combination of an element of $G^{*3}$ tell us the lengths of the three parts of the partition in order.  The elements $x_1,x_2,x_3$ of $G$ in the convex combination give us the labels of the intervals.}
        \label{fig:labels}
    \end{figure}

    For $j=1,\dots,m$ we define a function
    \begin{align*}
        f_j : K & \to \rr^p \\
        (A_1\dots, A_p) & \mapsto \left( y_1,\dots,y_p \right) \\
        y_u & = \sum_{i + t \equiv u \mod p} \mu^j_i(A_t).
    \end{align*}

The function $f_j$ is continuous.  The action of $G$ on itself induces a free action on $K$, which corresponds to relabelling the sets $A_1,\dots, A_p$ cyclically.  The space $\rr^p$ also has an action of $G$ induced by cycling the coordinates.  This action is free on $\rr^p\setminus\{0\}$.  With these two actions, the function $f_j$ is $G$-equivariant.

The way that $f_j$ is splitting the labeled measure $\mu_j$ corresponds to the distribution of the cookies as described in \cref{sec:intuition}.  As the value of $i$ increases, the measure $\mu^j_{i}$ is assigned to the next coordinate $y_u$ cyclically.  If we have a change from $A_t$ to $A_{t'}$, it corresponds to a cut as in \cref{sec:intuition}.  If the change happens on cookie $i$, it corresponds to a cookie being cut that has part of it given to child $i+t$ and part given to child $i+t'$.

Consider $W_p \subset \rr^p$ the orthogonal complement to $(1,1,\dots,1)$, consisting of all vectors whose sum of coordinates is $0$, and let $\pi : \rr^p \to W_p$ be the orthogonal projection onto $W_p$.  The space $W_p$ has an action of $\zz_p$, and $\pi$ is a $G$-equivariant function.  Note that $\pi \circ f_j = 0$ if and only if $y_1 = \dots = y_p$, which means that the $j$-th kind of frosting would be fairly split.

Now consider the function

\begin{align*}
    f&:  K \to (W_p)^{\oplus m} \\
    f&= (f_1,\dots, f_m)
\end{align*}

If $f$ has a zero, it must correspond to a set of cuts and a distribution such that each kind of frosting is evenly distributed among the children.  Let us assume for a contradiction that $f$ has no zero.

Denote by $S((W_p)^{\oplus m})$ the unit sphere in $(W_p)^{\oplus m}$.  The space $(W_p)^{\oplus m}$ has dimension $m(p-1)$, so the space $S((W_p)^{\oplus m})$ has dimension $m(p-1)-1$.  Moreover, the action of $\zz_p$ on $S((W_p)^{\oplus m})$ is free.  Consider the following map.

\begin{align*}
    g:  K &\to S((W_p)^{\oplus m}) \\
    g(\cdot)&= \frac{f(\cdot)}{\|f(\cdot)\|} 
\end{align*}

Due to Dold's theorem, the existence of $g$ shows that $K$ is at most $((p-1)m-2)$-connected.  What remains to show to reach a contradiction is that $K$ is at least $((p-1)m-1)$-connected.  The space $G_{*2}$ is path-connected, so it is $0$-connected.  Therefore, if $s$ is odd, we know that the connectedness of the $((s+1)/2)$-fold join $K=(G_{*2})^{*(s+1)/2}$ is at least $2((s+1)/2)-2 = s-1 = (p-1)m-1$.  If $s$ is even, then the connectedness of the $(s/2)$-fold join $(G_{*2})^{*s/2}$ is at least $2(s/2)-2 = s-2$.  Therefore, the connectedness of $K = (G_{*2})^{*s/2}*G$ is at least $(s-2)+(-1)+2 = s-1 = (p-1)m-1$, as we wanted.
\end{proof}

To finish the proof of \cref{thm:main-theorem}, we need to establish the following lemma.

\begin{lemma}
    If \cref{thm:main-theorem} holds for $r=a$ and for $r=b$, then it holds for $r=ab$.
\end{lemma}

\begin{proof}
    Assume we have $n$ cookies and that $a \ge b$.  We will split the children into $b$ groups of $a$ students each.  We first do a distribution for the $b$ groups.  Then, each of the small groups of $a$ students will append the pieces of cut cookies they received to one of their full cookies and do a distribution of their new set of cookies among themselves.

    First, we count the number of cuts we are making.  In the first distribution we make $(b-1)m$ cuts.  In each of the subsequent distributions we make $(a-1)m$ cuts.  Therefore, we make $(b-1)m + b(a-1)m = (ab-1)m$ cuts, as desired.  Let us now count how many full cookies each child receives.  After the first distribution, each group has at least
    \[
    k= \left\lfloor \frac{n-\frac{1}{2}(b-1)m}{b}\right\rfloor - \frac{1}{2}(b-1)m
    \]
    cookies.
    Note that
    \[
    k \ge \frac{n}{b} - (b-1)m
    \]
    Then, after the second round of cuts each child receives at least
    \[
    \left\lfloor \frac{k-\frac{1}{2}(a-1)m}{a}\right\rfloor - \frac{1}{2}(a-1)m \ge \frac{k}{a}-(a-1)m
    \]
    It suffices to show that

    \begin{align*}
        \frac{k}{a}-(a-1)m &\ge \frac{n-\frac{1}{2}(ab-1)m}{ab} - \frac{1}{2}(ab-1)m \\
                \frac{\frac{n}{b} - (b-1)m}{a}-(a-1)m &\ge \frac{n-\frac{1}{2}(ab-1)m}{ab} - \frac{1}{2}(ab-1)m \\
                \frac{ - (b-1)m}{a}-(a-1)m &\ge \frac{-\frac{1}{2}(ab-1)m}{ab} - \frac{1}{2}(ab-1)m \\
                \frac{\frac{1}{2}(ab-1)}{ab} + \frac{1}{2}(ab-1) &\ge \frac{ (b-1)}{a}+(a-1)
    \end{align*}

    Finally, if $b \ge 3$, we have $\frac{1}{2}(ab-1)\ge \frac{3}{2}a-\frac{1}{2}$, which would be greater than the right-hand side.   If $b=2$, we have to show that
\begin{align*}
    \frac{a-1/2}{2a}+a-1/2 &\ge \frac{1}{a}+a-1 \\
    \frac{1}{2}-\frac{1}{4a} +\frac{1}{2} &\ge \frac{1}{a} \\
    2a - 1 + 2a &\ge 4 \\
    2a &\ge 5.
\end{align*}

As long as $a \ge 3$, we have the result.  For the remaining case $a=b=2$ we can instead use \cref{thm:two-optimal} to estimate how many full cookies each student receives.

After the first cut each has at least

\[
k = \left\lfloor \frac{n-m}{2}\right\rfloor
\]
Therefore, after the second cut we have to show that
\[
\left\lfloor\frac{\left\lfloor\frac{n-m}{2}\right\rfloor-m}{2} \right\rfloor\ge \left\lfloor\frac{n-\frac{3m}{2}}{4}\right\rfloor-\frac{3m}{2},
\]
which can be checked directly.  The case $a=2,b=2$ is also a consequence of \cref{thm:powers-of-two}.
\end{proof}

\section{Remarks}\label{sec:remarks}

If $n+m$ is a multiple of $r$, then $\frac{n-(r-1)m}{r}$ is an integer, so we do not need the floor function in \cref{conj:main-conj}.  We first show the following lemma.

\begin{lemma}
    If \cref{conj:main-conj} holds for $r=a$ when $n+m$ is a multiple of $a$ and for $r=b$ when $n+m$ is a multiple of $b$, then it holds for $r=ab$ when $n+m$ is a multiple of $ab$.
\end{lemma}

\begin{proof}
    We follow the same subdivision argument as in the proof of \cref{thm:main-theorem}.  Assume $n+m$ is a multiple of  $ab$.  Since $n+m$ is a multiple of $b$, we can make the first subdivision exaclty.  Then, each group of $a$ children received $\frac{n-(b-1)m}{b}$ full cookies.  Note that $\left(\frac{n-(b-1)m}{b}\right)+m = \frac{n+m}{b}$, which is a multiple of $a$.  We can therefore make the second subdivision exactly.  As before, the subdivision argument gives us the correct number of cuts.  The number of full cookies that each child receives is at least
    \[
\frac{\frac{n-(b-1)m}{b}-(a-1)m}{a} = \frac{n-(b-1)m-b(a-1)m}{ab}= \frac{n-(ab-1)m}{ab},
\]
as we wanted to show.
\end{proof}

\begin{proof}[Proof of \cref{thm:powers-of-two}]
    Use induction on $t$ and \cref{thm:two-optimal} to prove the case when $n+m$ is a multiple of $r$.  If $n+m$ is not a multiple  of $r$, then add empty cookies until $n+m$ is a multiple of $r$.  If we added $x$ cookies, the number of full original cookies that each child receives is at least
\[
\left\lceil \frac{n-(r-1)m}{r}\right\rceil-x \ge \left\lceil \frac{n-(r-1)m}{r}\right\rceil-(r-1) = \left\lfloor \frac{n-(r-1)m}{r}\right\rfloor-(r-2).
\]
\end{proof}

In the introduction, we mentioned that there is a naive approach to \cref{conj:main-conj} with the test map scheme.  We describe it here to highlight the benefits of the construction we used.  As before, we place our cookies in the interval $[0,1]$ and use $\zz_r^{*(s+1)}$ to parametrize the partitions of $[0,1]$ using $s$ cuts where each interval has a label in $\zz_r$.  This time we do not attempt to do the alternating distribution and simply give to child $i$ the cookies in intervals with label $i-1$.  We do not need to represent the frosting as a labeled measure, we can simply use a measure $\mu_j$ for all the frosting of type $j$.

In addition to the $m$ measures $\mu_1, \dots, \mu_m$ induced by the different frostings, we give to each cookie an additional measure $\mu_{m+1}$ which has value $1$ on each cookie (i.e., the total cookie dough in the cookie).  We now seek to find a partition that divides fairly each of the $m+1$ measures.  As before, if there is no such fair partition, we can construct a continuous $\zz_r$-equivariant function 
\[
g:\zz_r^{*(s+1)}\to S((W_r)^{\oplus (m+1)})
\]
If for $r$ prime, to show that this map does not exist via Dold's theorem we require $r$ to be prime, and $s \ge (r-1)(m+1)$.  Each child receives a total amount $n/r$ of cookie dough.  The amount that comes from non-full cookies is at most $(r-1)(m+1)$, so each child receives at least $\lceil n/r \rceil - (r-1)(m+1)$ full cookies.

As a final note, consider the ways we cut cookies.  In the proof of \cref{thm:main-theorem}, we represented each cookie $C$ in an interval $I_C$ contained in $[0,1]$.  Then, each kind of frosting was represented by a uniform measure in $I_C$.  Therefore, when we cut a cookie and give its pieces to children $i$ and $j$, we give child $i$ the same percentage of each kind of frosting in $C$ and child $j$ the rest. However, we can choose any measures that are absolutely continuous with respect to the Lebesgue measure to represent the frostings and the proof holds.

For example, we could have uniform measures with disjoint supports in even smaller intervals.  This corresponds to first trying to give child $i$ all frosting of some kind before cutting out frosting of another kind.  The fact that, regardless of how we represent the cookies, the number of cuts and the guaranteed number of full cookies remain unchanged is surprising.  This is also the reason why we said in the introduction that if the cookies are subsets of $\rr^d$, the cuts can be made by hyperplanes.  We simply scale down and then project each cookie to the $x$-axis for the setup of the proof.  The inverse image of the cutting points in $[0,1]$ will give us the cutting hyperplanes in $\rr^d$.

Additionally, the order in which we place the cookies has an impact on the possible distributions we obtain.  It seems likely that this can be used to find a lower bound on the number of distributions achieving \cref{thm:main-theorem}.  Since \cref{thm:main-theorem} is not optimal, we have not attempted to bound the number of valid distributions.
\section{Acknowledgments}

I would like to thank Csaba Bir\'o, who told me about the case $r=2,m=1$ as a fun puzzle.  That case appeared in the 2008 Arany D\'aniel Mathematics Competition.


\end{document}